\documentclass[a4paper,10pt]{amsart}

\usepackage[T1]{fontenc}
\usepackage[utf8]{inputenc}
\usepackage{latexsym}
\usepackage{amsfonts}
\usepackage{amstext}
\usepackage{amsmath}
\usepackage{graphicx}
\usepackage{hyperref}
\usepackage{enumitem}
\usepackage{mathtools}

\usepackage{amssymb,euscript,graphicx}
\usepackage[ec]{aeguill}       
\usepackage{delarray, fancyhdr}
\usepackage[active]{srcltx}

\newcommand{\be}{\begin{enumerate}}
\newcommand{\ee}{\end{enumerate}}
\newcommand{\beqn}{\begin{eqnarray*}}
\newcommand{\eeqn}{\end{eqnarray*}}

\newcommand{\incl}[1][r]
      {\ar@<-0.2pc>@{^(-}[#1] \ar@<+0.2pc>@{-}[#1]}

\def\N{{\mathbb N}}

\def\Pc{{\mathcal P}}

\def\mgo{{\mathfrak m}}

\def\eps{\varepsilon}

\def\ov{\overline}
\newcommand{\cartesien}{\ar@{}[dr]|{\square}}

\newcommand{\double}{\ar@<2pt>[r] \ar@<-2pt>[r]}

\def\Ann{\hbox{\rm Ann}}

\def\Com{\hbox{\rm Com}\,}

\DeclareMathOperator{\edim}{edim}

\DeclareMathOperator{\Fit}{Fit}

\def\Frac{\hbox{\rm Frac}\,}

\def\Hom{\hbox{\rm Hom}}

\def\Id{{\rm Id}}
\def\Im{\hbox{\rm Im}\,}

\def\Ker{\hbox{\rm Ker}\,}

\DeclareMathOperator{\prof}{prof}

\DeclareMathOperator{\Spec}{Spec}

\def\Tor{{\rm Tor}\,}

	{\begin{pmatrix}}%
	{\end{pmatrix}}
	{\begin{vmatrix}}%
	{\end{vmatrix}}

	{\ensuremath{\left \{ \hskip -1.5 mm \begin{array}{#1@{\ =\ }l}}}%
	{\end{array}\right.}

	{\ensuremath{\left \{ \hskip -1.5 mm%
	 \begin{array}{#1@{\ }c@{\ }l}}}%
	{\end{array}\right.}

	{\ensuremath{\left \{ \hskip -1.5 mm \begin{array}{#1@{\quad}l}}}%
	{\end{array}\right.}

\newcommand{\details}[1]{}

\newtheorem{cor}[subsection]{Corollary}

\newtheorem{prop}[subsection]{Proposition}

\newtheorem{thm}[subsection]{Theorem}

\newtheorem{defi}[subsection]{Definition}

\newtheorem{lem}[subsection]{Lemma}

\newtheorem{remarque}[subsection]{Remark}

\newtheorem{exemple}[subsection]{Example}

\newtheorem{question}[subsection]{Question}

\usepackage[all,dvips]{xy}
\input xy
\xyoption{all}
\def\lto{%
\xymatrix{\ar[r] &}}

\DeclareMathOperator{\Jac}{Jac}
\DeclareMathOperator{\pd}{pd}
\renewcommand{\prof}{\rm depth}

\title{Independent sequences and freeness criteria}
\date{\today}

\begin{document}

\author{Sylvain Brochard}
\email{sylvain.brochard@umontpellier.fr}
\address{
IMAG, Université de Montpellier, CNRS, Montpellier, France
}


  \begin{abstract}
Let $M$ be a module over a Noetherian local ring $A$. We study $M$-independent sequences of elements of $\mgo_A$ in the sense of Lech and Hanes. The main tool is a new characterization of the $M$-independence of a sequence in terms of the associated Koszul complex.  As applications, we give a result in linkage theory, a freeness criterion for $M$ in terms of the existence of a strongly $M$-independent sequence of length $\edim(A)$, and another freeness criterion inspired from the patching method of Calegari and Geraghty for balanced modules in their 2018 paper.
 \end{abstract}
  
\maketitle

\section{Introduction}

Independent sequences were introduced by Lech in~\cite{Lech_Inequalities_Related_To}. A sequence $x_1,\dots, x_n$ of elements of a commutative ring $A$ is called \emph{independent} in $A$ if, for every sequence $(a_1,\dots, a_n)\in A^n$, the equality $\sum x_ia_i=0$ implies $a_1,\dots, a_n\in (x_1,\dots, x_n)$. This notion of independence was used by Vasconcelos~\cite{Vasconcelos_Ideals_Generated_By} to characterize ideals generated by regular sequences: a sequence is regular if and only if it is independent and generates an ideal of finite projective dimension (see~\ref{prop:caract_ideaux_reguliers_Vasconcelos}). Besides regular sequences, another important source of examples is the following: any minimal system of generators of the maximal ideal $\mgo_A$ of a Noetherian local ring $A$ is independent.
 Independent sequences were further explored by Hanes in~\cite{Hanes_Length_Approximations_For}, who gave a lower bound for the colength of an ideal generated by an independent sequence, and used it to give partial results for Lech's conjecture about multiplicities in flat couples of local rings. Recently, Meng defined in~\cite{Meng_Strongly_Lech_Independent} the notion of \emph{strongly Lech-independent} sequence and, as Hanes, used it to derive some particular case of Lech's conjecture.

Both notions naturally generalize to notions of independence (resp. strong independence) relatively to a module~$M$ over~$A$ (see~\ref{def:Mindep_sequence} and~\ref{def:strong_M_indep}). Any $M$-regular sequence is (strongly) $M$-independent, but there are in general many more $M$-independent sequences, an extreme case being that of modules over an Artin local ring, where there are no regular sequences at all. Despite their ubiquity, $M$-independent sequences have been relatively unexploited so far. One possible reason for this is that there are few results about them in the literature.

The main result of this paper concerning $M$-independent sequences is the following Theorem. It was previously known only when the matrix $(w_{ij})$ that appears in the statement is diagonal (\cite[Lemma 3]{Lech_Inequalities_Related_To} and~\cite[Lemma 2.1]{Hanes_Length_Approximations_For}), or when the sequence $x_1,\dots, x_n$ is regular~\cite[Proposition 1.2]{Simon_Strooker_Wiebe}.

\begin{thm}
\label{thm:intro}
 Let $A$ be a ring and let $M$ be an $A$-module. Let $x_1,\dots, x_n$ and $u_1,\dots, u_n$ be sequences of elements of $A$, that generate ideals $J_x$ and $J_u$. Assume that the $x_i$'s form an $M$-independent sequence and that for any $i$ there exist elements $w_{ij}\in A$ such that $x_i=\sum w_{ij}u_j$ (i.e. $J_x\subset J_u$). Then:
 \begin{enumerate}
  \item The sequence $u_1,\dots, u_n$ is $M$-independent.
  \item If $\Delta$ denotes the determinant of the matrix $(w_{ij})$, we have the following equalities of submodules of $M$:
  \begin{align*}
 (J_xM:J_u)=(J_x+(\Delta)) M \\
(J_xM: (J_x+(\Delta))) =J_u M
\end{align*}
\item If $M$ is nonzero and of finite type, and if $J_u\subset \Jac(A)$, then $\Delta \notin J_x$.
\end{enumerate}
\end{thm}

A more complete statement is given in~\ref{thm:liaison_suites_Mindep}. Assertion (2) can be viewed as a result in linkage theory. Indeed, if $n=\edim(A)$, the ring $A/J_x$ is a complete intersection by~\ref{thm:freeness_criterion_from_Mindep}, and assertion (2) then says that the ideals $J_u$ and $J_x+(\Delta)$ are linked by $J_x$ over the module $M$ (see~\cite{Peskine_Szpiro_Liaison_Des_Varietes}, \cite{Martsinkovsky_Strooker_Linkage_Of_Modules} or \cite{Jahangiri_Sayyari_Linkage_Of_Ideals}). The main tool to prove~\ref{thm:intro} is a characterization of the  $M$-independence of a sequence $x$ in terms of the Koszul complex $K(x)$, see~\ref{prop:Mindep_en_termes_Koszul} and its corollaries~\ref{cor:pseudo-homotopies} and \ref{cor:homotopies}.

\subsection*{Applications.}
Let $A$ be a Noetherian local ring and let $M$ be a finite type $A$-module of finite projective dimension. Then the depth of $M$ and its projective dimension are related by the following formula (Auslander-Buchsbaum):
\[
 \prof_A(M)+\pd_A(M)=\prof(A).
\]
A useful consequence of this is that $M$ is free over $A$ if and only if it has a regular sequence of the maximal possible length $\prof(A)$. Using~\ref{thm:intro}, we get a similar statement involving $M$-independent sequences instead of $M$-regular sequences. Before stating it, note that it follows from~\ref{thm:liaison_suites_Mindep} that, if $A$ is a Noetherian local ring, $\edim(A)$ is the maximal possible length of an $M$-independent sequence.

\begin{thm}
 \label{thm:intro_freeness}
 Let $A$ be a Noetherian local ring and $M$ an $A$-module of finite type.
 \begin{enumerate}
  \item $M$ is free over $A$ if and only if there exists a strongly $M$-independent sequence of length~$\edim(A)$ in $\mgo_A$.
  \item $M/\mgo_A^2M$ is free over $A/\mgo_A^2$ if and only if there exists an $M$-independent sequence of length~$\edim(A)$ in $\mgo_A$.
 \end{enumerate}
\end{thm}

\noindent
Theorem~\ref{thm:intro_freeness} is an immediate consequence of~\ref{prop:free_implies_maxid_indep} and \ref{thm:freeness_criterion_from_Mindep}. Note that we do not need to assume that $M$ has finite projective dimension. To see how this freeness statement can be applied, let us give a new and more transparent proof of de Smit's conjecture (see~\cite{Brochard_DeSmit2}, and~\cite{Brochard_Iyengar_Khare_A_Freeness_Criterion} for a more general statement).

\begin{cor}
 \label{cor:desmit}
 Let $\varphi : A\to B$ be a flat local morphism of Noetherian local rings with $\edim(A)\geq \edim(B)$ and let $M$ be a $B$-module of finite type that is free over~$A$. Then $M$ is free over $B$ and $\edim(A)=\edim(B)$.
\end{cor}
\begin{proof}
 Since $M$ is free over $A$, there exists a strongly $M$-independent sequence $x_1,\dots, x_n$ of elements of $A$ with $n=\edim(A)$. By~\ref{prop:permanence_indep} the sequence $\varphi(x_1),\dots, \varphi(x_n)$ is a strongly $M$-independent sequence in $\mgo_B$. Since $n\geq \edim(B)$, it follows from~\ref{thm:intro_freeness} that $M$ is $B$-free and from~\ref{thm:liaison_suites_Mindep} that $\edim(A)=\edim(B)$. 
\end{proof}

De Smit's conjecture appeared in the context of the proof of Fermat's Last Theorem, to avoid the patching method.
This patching method has been generalized and adapted to many other contexts in modularity lifting problems. A recent example is the freeness criterion given by Calegari and Geraghty in~\cite[\S 2]{Calegari_Geraghty}. It seems natural to wonder if we can dispense with patching in their work, as de Smit's statement allows to do in the proof of Fermat's Last Theorem.
Our main result in this direction is the following (see~\ref{thm_VP}), which means that it is possible to avoid patching, at the cost of a complete intersection assumption.  We refer to the introduction of Section~4 for more details about this question.

\begin{thm}
 \label{thm:intro_VP}
Let $\varphi : A\to B$ be a finite local morphism of Noetherian local rings such that $B/\mgo_A B$ is a complete intersection. 
Let $M$ be a $B$-module which is of finite type over $A$.
Assume that
\[
 \dim_k\Tor_A^1(M,k)\leq (\edim(A)-\edim(B)).\dim_k\Tor_A^0(M,k).
\]
Then $M$ is free over $B$, and the above inequality is an equality.
\end{thm}

\noindent

\subsection*{Notations.}
If $A$ is a commutative ring, $\Jac(A)$ denotes its Jacobson radical. If $M$ is an $A$-module, its annihilator is denoted by $\Ann_A(M)$. If $N$ is a submodule of~$M$, the annihilator of $M/N$, i.e. the ideal of elements $a\in A$ such that $am\in N$ for all $m\in M$, is denoted by $(N:M)$. If $I$ is an ideal of $A$, the submodule $\{m\in M\ |\  \forall a\in I, am\in N\}$ is denoted by $(N:I)$. If $M$ is an $A$-module of finite type, we denote by $\Fit_A(M)$ its initial Fitting ideal.

If $A$ is a Noetherian local ring, we 
denote by $\mgo_A$ its maximal ideal and by~$\kappa(A)$ its 
residue field. The embedding dimension of $A$, i.e. the minimal number of 
generators of $\mgo_A$, is denoted by $\edim(A)$. The minimal number of generators of an $A$-module~$M$ is denoted by
$\mu_A(M)$. We have 
$\mu_A(M)=\dim_{\kappa(A)}(M/\mgo_AM)$. In particular, $\edim(A)=\mu_A(\mgo_A)$.

\subsection*{Acknowledgments.} I warmly thank Vincent Piloni who pointed out to me the article~\cite{Calegari_Geraghty}, and raised the central question of 
section~\ref{section:balanced_desmit}.

\section{Elementary properties of \texorpdfstring{$M$}{M}-independent sequences}
\label{section:indep_seq}

The following definition of an $M$-independent sequence is due to Lech \cite[p. 77]{Lech_Inequalities_Related_To} and Hanes \cite[Definition 2.1]{Hanes_Length_Approximations_For}.

\begin{defi}
\label{def:Mindep_sequence}
Let $A$ be a ring and let $M$ be an $A$-module. A 
sequence $x=(x_1, \dots, x_n)$ of elements of $A$ is $M$-independent if, for 
any $m_1, \dots, m_n\in M$, the relation $\sum x_im_i=0$ implies $m_i\in (x_1, 
\dots, x_n)M$.
\end{defi}

\noindent
Equivalently, if $I$ denotes the ideal generated by $x_1,\dots, x_n$, the sequence $x$ is $M$-independent if and only if the natural surjection 
\[
 \varphi_x : (M/IM)^n \lto IM/I^2M
\]
defined by $\varphi_x(m_1,\dots, m_n)=\sum x_im_i$ is an isomorphism. If $A$ is a Noetherian local ring and if $M=A$, it follows from Nakayama's Lemma that the sequence $x$ is $A$-independent if and only if $I/I^2$ is a free $(A/I)$-module. In particular this does not depend on the choice of the generating sequence $x$ for $I$. Meng defined the notion of strong independence as follows~\cite[Definition 3.1]{Meng_Strongly_Lech_Independent}: the ideal $I$ is said to be strongly independent if $I^i/I^{i+1}$ is free over $A/I$ for any $i\geq 1$. The following definition will be convenient to give another rephrasing of Definition~\ref{def:Mindep_sequence} and to generalize the strong independence to modules.

\begin{defi}
 Let $A$ be a Noetherian local ring, $M$ an $A$-module and $I$ an ideal of~$A$. 
Let $(x_1, \dots, x_n)$ be a minimal system of generators of $I$. We denote by 
$R_I(M)$ the submodule of $M$ generated by the elements $m_i\in M$ for all 
relations
\[
 x_1m_1+\dots+x_nm_n=0\, .
\]
(Note that, by~\cite[4.1]{Brochard_Mezard_ConjDeSmit}, $R_I(M)$ does 
not depend on the choice of $(x_1, \dots, x_n)$.)
\end{defi}

\noindent
By definition, a sequence $x_1,\dots, x_n$ is $M$-independent if and only if $R_I(M)\subset IM$, where $I=(x_1,\dots, x_n)$. We shall sometimes say that the ideal $I$ is $M$-independent. Note also that if $A$ is a Noetherian local ring and if $x_1,\dots, x_n\in \mgo_A$ is an $M$-independent sequence for some nonzero $A$-module $M$ of finite type, then by Nakayama $x_1,\dots, x_n$ is necessarily a \emph{minimal} system of generators of the ideal $(x_1,\dots, x_n)$. 
\details{
Preuve : Si $x$ n'est pas minimal, on peut extraire de $x$ un système minimal de générateurs. Alors l'un des $x_i$, par exemple $x_n$, est combinaison linéaire des autres. Par $M$-indépendance (et parce que $I\subset \mgo_A$), on en déduit que $M\subset \mgo_AM$, d'où $M=0$ par Nakayama.
}

\begin{defi}
\label{def:strong_M_indep}
 Let $A$ be a Noetherian local ring and $M$ an $A$-module. We say that an ideal $I$ of $A$ is strongly $M$-independent if $R_{I^i}(M)\subset IM$ for any $i\geq 1$. A sequence is called strongly $M$-independent if it generates a strongly $M$-independent ideal. 
\end{defi}

Independent sequences behave well through morphisms of rings.

\begin{prop}
\label{prop:permanence_indep}
 Let $\varphi : A\to B$ be a local morphism of Noetherian local rings and let $M$ be a $B$-module.
\begin{enumerate}
\item A sequence of elements of $A$ is $M$-independent if and only if its image by~$\varphi$ is $M$-independent.
\item Assume that $\varphi$ is flat. Then a sequence of elements of $A$ is strongly $M$-independent if and only if its image by $\varphi$ is strongly $M$-independent.
\end{enumerate}
\end{prop}
\begin{proof}
(1) is obvious on the definition of $M$-independence. (2) also follows from the definition, and from the fact that for any ideal $I$ of $A$, if $x_1,\dots, x_n$ is a minimal system of generators of $I$, then $\varphi(x_1),\dots, \varphi(x_n)$ is a minimal system of generators of the ideal $IB$ of $B$. The latter fact holds because $B$ is flat over $A$.
\end{proof}

As mentionned above, any $M$-regular sequence is strongly $M$-independent: this is an immediate consequence of the fact that any $M$-regular sequence is $M$-quasi-regular, see e.g.~\cite[\href{https://stacks.math.columbia.edu/tag/061M}{10.69.2, Tag 061M}]{Stacks_Project}.
\details{
Remarque : voici une preuve très élémentaire du fait que régulier implique $M$-indépendant (ce qui est plus faible que strongly $M$-independent). Let $(x_1, \dots, x_n)$ be an $M$-regular sequence. If $n=0$ there is nothing to prove. Arguing by induction, we may assume that any $M$-regular sequence of length $n-1$ is $M$-independent. In particular the sequence $(x_1,\dots, x_{n-1})$ is $M$-independent. Let $x_1m_1+\dots+x_nm_n=0$ be a relation in $M$. Since $x_n$ is a nonzerodivisor on $M/(x_1,\dots, x_{n-1})M$ we can write $m_n=x_1m'_1+\dots+x_{n-1}m'_{n-1}$ for some $m'_i\in M$. Then $\sum_{i=1}^{n-1}x_i(m_i+x_nm'_i)=0$ and the $M$-independence of the sequence $(x_1,\dots, x_{n-1})$ implies that $m_i\in (x_1,\dots, x_n)M$ for all $i$. 
}
To illustrate further the connection between independent sequences and regular sequences, let us recall the following result of Vasconcelos.

\begin{prop}[{\cite[Cor.1 p.311]{Vasconcelos_Ideals_Generated_By}}]
\label{prop:caract_ideaux_reguliers_Vasconcelos}
 Let $A$ be a Noetherian local ring and let $x=(x_1, \dots, x_n)$ be a sequence 
of elements of $\mgo_A$. The following are equivalent:
\begin{enumerate}
 \item  The sequence $x$ is regular.
\item The ideal $J_x=(x_1,\dots, x_n)$ has finite projective dimension, and the 
sequence $x$ is $A$-independent.
\end{enumerate}
\end{prop}

\details{
La généralisation qui pourrait sembler naturelle, à savoir qu'en supposant $\pd_A(M)<+\infty$, une suite $(x_1,\dots, x_n)$ serait $M$-régulière ssi elle était $M$-indépendante et $\pd_A(J_xM)<+\infty$, est fausse. Voir l'exemple ci-dessous.
}

\noindent
This result implies in particular that the notions of $A$-independent sequence and of regular sequence coincide on a regular local ring $A$. This fails for $M$-independent sequences in general as the following example shows.

\begin{exemple}\rm
\label{ex:longer_Mindep_sequence}
 Let $A$ be a discrete valuation ring with uniformizer $\pi$ and let $M=A/(\pi^2)$. Then $(\pi)$ is $M$-independent, but it cannot be regular since $\prof_A(M)=0$.
\end{exemple}

In the next two propositions, we explore the relashionship between the existence of $M$-independent ideals and the freeness of $M$.

\begin{prop}
 \label{prop:free_implies_maxid_indep}
Let $A$ be a Noetherian local ring and let $M$ be an $A$-module.
\begin{enumerate}
 \item If $M/\mgo_A^2M$ is free over $A/\mgo_A^2$, then $\mgo_A$ is $M$-independent.
 \item If $M$ is free over $A$, then for any ideal $I$ of $A$, we have $R_I(M)\subset \mgo_AM$. In particular, the maximal ideal $\mgo_A$ is strongly $M$-independent.
\end{enumerate}
\end{prop}
\begin{proof}
(1) is an immediate consequence of (2). Let us prove (2).
Let $y_1, \dots, y_{\ell}$ be a minimal system of 
generators of $I$ and let
\[
 y_1m_1+\dots+y_{\ell}m_{\ell}=0
\]
be a relation in $M$. We have to prove that $m_j\in \mgo_AM$ for any $j$.
Let $(e_i)$ be a basis of $M$ and let~$\lambda_{ij}$ be the coefficient of 
$m_j$ along $e_i$. Then for any index $i$ we have
$\sum_j y_j\lambda_{ij}=0$. But the projections of the elements 
$y_j$ form a $\kappa(A)$-basis of $I/\mgo_AI$, hence 
$\lambda_{ij}\in\mgo_A$ for all $i,j$, which proves that $m_j\in\mgo_AM$.
\end{proof}

\begin{prop}
\label{prop:indep_implies_freeness}
 Let $A$ be a Noetherian local ring, $M$ an $A$-module and $I\subset \mgo_A$ an ideal of $A$. Assume that $M/IM$ is free over $A/I$.
\begin{enumerate}
\item If $I$ is $M$-independent, then $M/I^2M$ is free over $A/I^2$.
\item If $I$ is strongly $M$-independent, then $M$ is free over $A$.
\end{enumerate}
\end{prop}
\begin{proof}
 We first prove by induction on $n\in \N$ that, if $R_{I^i}(M)\subset IM$ for any $i\in \{0,\dots, n\}$, then $M/I^{n+1}M$ is free over $A/I^{n+1}$. For $n=0$ there is nothing to prove. Assume that it holds for $n-1$, with $n\geq 1$. 
 Let $(e_i)_{i\in I}$ be a family of elements of~$M$, the images of which form an $A/I^n$-basis of $M/I^nM$. Then by Nakayama the $e_i$'s generate $M/I^{n+1}M$ over $A/I^{n+1}$. Let us prove that they 
actually form a basis. Let 
$(\lambda_i)$ be a finite family of elements of~$A$ such that $\sum\lambda_ie_i\in I^{n+1}M$. Since the images of the $e_i$'s form a basis of $M/I^nM$, we already know that $\lambda_i\in I^n$ for all~$i$, hence we can write $\lambda_i=\sum_{j=1}^\ell x_j\lambda_{ij}$ with $\lambda_{ij}\in A$, where $x_1,\dots, x_\ell$ is a minimal system of generators of $I^n$. On the other hand, since $\sum \lambda_ie_i\in I^{n+1}M$, there exist elements $m_j\in IM$ such that $\sum \lambda_ie_i=\sum_{j=1}^\ell x_jm_j$. We get the equality
\[
 \sum_{j=1}^\ell x_jm_j=\sum_{j=1}^\ell x_j\left(
    \sum_{i\in I}\lambda_{ij}e_i
 \right).
\]
Since $R_{I^n}(M)\subset IM$, it follows that for any $j$, the element $m_j-\sum_{i\in I}\lambda_{ij}e_i$ belongs to $IM$, hence so does $\sum_{i\in I}\lambda_{ij}e_i$. Since the $e_i$'s form an $A/I$-basis of $M/IM$, this implies that $\lambda_{ij}\in I$ for all $i$ and $j$, hence $\lambda_i\in I^{n+1}$. This concludes the proof of the claim. Now $(1)$ is the case $n=1$, and $(2)$ is a consequence of~\cite[IV, 5.6]{SGA1}.
\end{proof}

\details{
\begin{exemple}\rm
 If $I\subset J$ are two ideals of $A$, the submodules $R_I(M)$ and $R_J(M)$ are not comparable in general. For example let $A=k[x,y]/(x^4,y^4,x^2-xy, y^2-xy)$. Then we have
 \[
  R_{\mgo_A^i}(A)=\left\{
  \begin{array}{ll}
   \{0\} & \textrm{ if } i=0 \textrm{ or } i\geq 4, \\
   \mgo_A & \textrm{ if } i=1 \textrm{ or } i=3, \\
   (x-y)+\mgo_A^2 & \textrm{ if } i=2.
  \end{array}
  \right.
 \]
Les éléments suivants forment une $k$-base de $A$ : $1, x, y, x^2, x^3$. On a les systèmes minimaux suivants de générateurs :
\[
\mgo_A^0=(1) \quad \mgo_A=(x,y) \quad \mgo_A^2=(x^2) \quad \mgo_A^3=(x^3),
\]
et $\mgo_A^4$ est nul (engendré par l'ensemble vide).
Si $1.m=0$ alors $m=0$, d'où $R_{\mgo_A^0}(A)=\{0\}$. Si $xf_1+yf_2=0$, en réduisant modulo $\mgo_A^2$ on voit que $f_1, f_2\in \mgo_A$, donc $R_{\mgo_A}(A)\subset \mgo_A$. L'inclusion réciproque vient de $xy-yx=0$. Soit $f=a+bx+cy+dx^2+ex^3\in A$. Si $x^2f=0=ax^2+bx^3+cx^2y=ax^2+(b+c)x^3$, alors $a=0$ et $b+c=0$, d'où l'inclusion $R_{\mgo_A^2}(A)\subset (x-y,x^2)$. L'inclusion réciproque est immédiate. Si $x^3f=0=ax^3$ alors $a=0$ d'où $R_{\mgo_A^3}(A)\subset \mgo_A$. La réciproque est claire car $x^4=0$ et $x^3y=0$. 
\end{exemple}
}

\begin{remarque}\rm
 To conclude this section, let us list a few other elementary facts about independence. Since none of these facts will be used in this paper, we leave the (elementary) proofs to the reader.
 \begin{enumerate}
  \item Let $I$ be an ideal of $A$ that is contained in the Jacobson radical of $A$, such that $IM$ is a finite type $A$-module, and let $(x_1, \dots, x_n)$ be a 
nonempty sequence of elements of $I$ which is $M/I^2M$-independent. Then the 
sequence is $M$-independent, and $IM=(x_1, \dots, x_n)M$.
\details{
Preuve : For any $m\in IM$, we have $x_1m=0$ 
in the quotient $M/I^2M$, hence $m\in (x_1, \dots, x_n)M+I^2M$, so 
that $IM\subset (x_1, \dots, x_n)M+I^2M$. By Nakayama's Lemma this implies that $IM=(x_1, \dots, x_n)M$. Now 
if we have a relation $\sum x_im_i=0$ then we have the same relation in the 
quotient $M/I^2M$ hence $m_i\in (x_1, \dots, x_n)M$ modulo $I^2M$. Since 
$I^2M\subset (x_1,\dots, x_n)M$ the result follows.
}
\item  Let $A$ be a ring and let $x,y$ be two sequences of elements of $A$, that generate ideals $J_x$ and $J_y$. Let $M$ be an $A$-module.
 \begin{enumerate}
  \item The sequence $(x,y)$ is $M$-independent if and only if $x$ is $(M/J_yM)$-independent and $y$ 
is $(M/J_xM)$-independent.
\item If $J_x\subset J_y$ and $y$ is $(M/J_xM)$-independent, then $y$ is $M$-independent.
 \end{enumerate}
 \details{
 Preuve :
 (a) Assume that $(x,y)$ is $M$-independent. Let us prove that $x$ is 
$M/J_yM$-independent. Let $\sum x_im_i=0$ be a relation in $M/J_yM$, i.e. $\sum 
x_im_i\in J_yM$ so that there exist elements $m'_j$ such that $\sum x_im_i=\sum 
y_jm'_j$ in $M$. Since $(x,y)$ is $M$-independent, it follows that for all $i$, 
$m_i\in (J_x+J_y)M$, hence the image of $m_i$ in $M/J_yM$ belongs to 
$J_x(M/J_yM)$, as desired. Similarly $y$ is $M/J_xM$-independent.
Conversely assume that $x$ is $(M/\!J_yM)$-independent and $y$ 
is $(M/J_xM)$-independent. Let $\sum x_im_i+\sum y_jm'_j=0$ be a relation in $M$. 
Then $\sum x_im_i=0$ in $M/J_yM$, hence $m_i\in J_xM+J_yM$, and similarly 
$m'_j\in J_xM+J_yM$.

(b) Let $\sum y_jm_j=0$ be a relation in $M$. The $M/J_xM$-independence of $y$ implies that for any $j$ we have $m_j\in J_xM+J_yM=J_yM$.
 }
 \end{enumerate}
\end{remarque}

\section{\texorpdfstring{$M$}{M}-independent sequences and the Koszul complex}

As for regular sequences, the $M$-independence can be detected by the Koszul 
complex. Let us first recall notations and basic facts about the Koszul 
complex. Given a sequence 
$x_1, \dots, x_n$ of elements of $A$, let $E$ be the free module $A^n$ and let 
$(e_1,\dots, e_n)$ be its canonical basis. Then the Koszul complex on $x$, denoted by~$K(x)$, is the 
complex whose degree $\ell$ component $K_\ell(x)$ is the free module $\Lambda^\ell E$ 
with basis $\{e_{i_1}\wedge\dots \wedge e_{i_\ell}\}$ with $i_1<\dots<i_\ell$, and 
whose differential $d_\ell^x : K_\ell(x)\to K_{\ell-1}(x)$ is given by the following formula, where the hat means that the term $e_{i_j}$ is omitted.
\[
 d_\ell^x(e_{i_1}\wedge\dots \wedge e_{i_\ell})
=\sum_{j=1}^{\ell}(-1)^{j-1}x_{i_j}e_{i_1}\wedge \dots \wedge
\widehat{e_{i_j}} \wedge \dots \wedge e_{i_\ell}
\]
If $M$ is an $A$-module, we denote by $K(x,M)$ the complex 
$K(x)\otimes_A M$, and by~$d^{x,M}_\ell$ its $\ell$th differential. 
Since $\Lambda^\ell E=0$ for $\ell>n$, this complex has length~$n$. If~$A$ is 
local, it is well-known that the sequence~$x$ is $M$-regular if and only if the 
Koszul complex $K(x,M)$ is exact, except at its extreme right where we 
have $H_0(K(x,M))\simeq M/J_xM$ where $J_x$ denotes the ideal of $A$ 
generated by $x_1, \dots, x_n$. For $M$-independence we have the following 
analogous characterization.

\begin{prop}
\label{prop:Mindep_en_termes_Koszul}
 Let $A$ be a ring and let $x=(x_1, \dots, x_n)$ be a sequence of elements of 
$A$. Let $M$ be an $A$-module.
\begin{enumerate}
 \item If $x$ is $M$-independent, then for any $\ell\in \{1, \dots, n\}$ we 
have
\[
 \Ker d^{x,M}_{\ell} \subset J_x.K_{\ell}(x,M).
\]
More generally, for any submodule $N\subset M$, we have
\[
 (d^{x,M}_{\ell})^{-1}(J_xK_{\ell-1}(x,N)) \subset 
K_{\ell}(x,N)+J_xK_{\ell}(x,M)
\]
\item Conversely, assume that
\[
 \Ker d^{x,M}_{1}\subset J_x.K_{1}(x,M),
\]
then the sequence $x$ is $M$-independent.
\end{enumerate}
\end{prop}
\begin{proof}
(1) Let $e_1, \dots, e_n$ be the canonical basis of $A^n$. Let $I_\ell$ denote 
the set of $\ell$-uplets $(i_1, \dots, i_\ell)$ such that $1\leq i_1<\dots 
<i_\ell \leq n$. For $i=(i_1, \dots, i_\ell)$ let $f_i=e_{i_1}\wedge \dots 
\wedge e_{i_\ell}$. Then $(f_i)_{i\in I_\ell}$ is a basis of 
$\Lambda^{\ell}(A^n)$. We use this basis to identify $K_\ell(x)$ with 
$A^{I_\ell}$ and $K_\ell(x,M)$ with $M^{I_\ell}$. Now let $m=(m_i)_{i\in 
I_\ell}$ be an element of $K_\ell(x,M)$ such that $d^{x,M}_\ell(m)\in 
J_xK_{\ell-1}(x,N)$. In particular for any $i=(i_1, \dots, i_{\ell-1})\in 
I_{\ell-1}$ the component of $d^{x,M}_\ell(m)\in M^{I_{\ell-1}}$ with label $i$ 
belongs to $J_x N$. But this component can be written:
\[
\sum_{j=1}^{i_1-1} x_jm_{(j,i_1,\dots, i_{\ell-1})}
-\sum_{j=i_1+1}^{i_2-1} x_jm_{(i_1,j,i_2,\dots, i_{\ell-1})}
+\dots
+(-1)^{\ell-1}\sum_{j=i_{\ell-1}+1}^{n} x_jm_{(i_1,\dots, i_{\ell-1},j)}
\]
Using the $M$-independence of the sequence $x$ this implies that for any $i\in 
I_{\ell}$, $m_i\in N+J_xM$. The result follows.

(2) Identifying $K_1(x,M)$ with $M^n$, the morphism $d_1^{x,M}$ maps
$(m_1, \dots, m_n)$ to $\sum x_im_i$, so that the inclusion $\Ker 
d^{x,M}_{1}\subset J_x.K_{1}(x,M)$ is by definition equivalent to the 
$M$-independence of $x$.
\details{Brouillon et détails du calcul en C6 p.60v.} 
\end{proof}

\begin{cor}
\label{cor:pseudo-homotopies}
 Let $M$ be an $A$-module and $x$ an $M$-independent sequence. Then for any 
morphism $f : N \to M$ of $A$-modules, and any morphism $\varphi : 
K_{\ell-1}(x)\to M$ such that
\(
 \varphi\circ d^{x}_\ell=f\circ \mu_{\ell}
\)
for some morphism $\mu_{\ell} \in J_x\Hom(K_{\ell}(x),N)$, there exists a 
morphism $h_{\ell-1} : K_{\ell-1}(x)\to N$ such that $\varphi=f\circ 
h_{\ell-1}$ modulo $J_xM$.
\[
\xymatrix{
K_{\ell}(x) \ar[r]^{d^{x}_{\ell}} \ar[d]_{\mu_{\ell}} &
K_{\ell-1}(x) \ar[d]^{\varphi} \ar@{.>}[ld]^{\exists h_{\ell-1}}\\
N\ar[r]_{f} &M
}
\]
\end{cor}
\begin{proof}
 Using the self-duality of the Koszul complex (see 
e.g.~\cite[17.15]{Eisenbud_Commutative_Algebra_With}) there is an isomorphism 
of complexes $K(x,M)\simeq \Hom(K(x), M)$. Hence we can 
apply~\ref{prop:Mindep_en_termes_Koszul} (1) to the complex $\Hom(K(x), M)$ and 
the submodule $N':=\Im f\subset M$. Since $\varphi\circ d^{x}_{\ell}\in 
J_x\Hom(K_{\ell}(x),N')$, it follows that $\varphi$ belongs to 
$\Hom(K_{\ell-1}(x),N')+J_x\Hom(K_{\ell-1}(x),M)$, in other words there is a 
morphism $h : K_{\ell-1}(x)\to N'$ 
such that $\varphi=h$ modulo $J_xM$. Since $K_{\ell-1}(x)$ is free, $h$ lifts
to a morphism $h_{\ell-1} : K_{\ell-1}(x)\to N$ such that $\varphi=f\circ 
h_{\ell-1}$ modulo $J_xM$.
\end{proof}

\begin{cor}
\label{cor:homotopies}
Let $M_{\bullet}$ be a complex of $A$-modules and let $x=(x_1,\dots, x_n)$ be a sequence of elements of $A$ that is $M_\ell$-independent for any $\ell$. Then any 
morphism of complexes $\varphi : K(x)\to M_{\bullet}$ such that $\varphi_n : K_n(x)\to M_n$ takes its values in $J_xM_n$, is homotopic to a morphism of complexes with values in $J_xM_{\bullet}$.
\end{cor}
\begin{proof}
 Let $f_i : M_i\to M_{i-1}$ denote the differential of $M_{\bullet}$. Using~\ref{cor:pseudo-homotopies}, we can construct inductively morphisms $h_{n-1}, \dots, h_0$ with $h_i : K_i(x)\to M_{i+1}$ such that for any $i$ the morphism $\varphi_i-f_{i+1}h_i$ takes its values in $J_xM_i$. Then the collection of morphisms $\psi_i:=\varphi_i-f_{i+1}h_i-h_{i-1}d_i^x$ defines a morphism of complexes $\psi : K(x)\to M_{\bullet}$ that is homotopic to $\varphi$ and takes its values in $J_xM_{\bullet}$.
\end{proof}

The following Theorem improves~\cite[1.2]{Simon_Strooker_Wiebe}. In the statement and its proof, if $x=(x_1,\dots, x_n)$ is a sequence of elements, we will still denote by the symbol $x$ the row matrix $(x_1,\dots, x_n)$, and by $x^T$ its transpose.

\begin{thm}
 \label{thm:liaison_suites_Mindep}
Let $A$ be a ring and $M$ an $A$-module. Let $x=(x_1, \dots, x_n)$ and $u=(u_1, \dots, u_n)$ be two sequences of elements of $A$, that generate ideals $J_x\subset J_u$. Assume that $x$ is $M$-independent. Then:
\begin{enumerate}
 \item The sequence $u$ is $M$-independent as well.
 \details{
 If moreover $W$ has its entries in $J_u$, then $J_u$ is also $M/J_xM$-independent. En effet : si $u.\mu=\in J_xM$, alors il existe $\mu'$ tel que $u.\mu=x.\mu'=uW.\mu'$, d'où $u.(\mu-W\mu')=0$ donc $\mu-W\mu'$ est à coefficients dans $J_uM$ par $M$-independance de $u$, donc $\mu$ aussi.
 }
\item Let $W$ be any square matrix with entries in $A$ 
such that $x=uW$ (since $J_x\subset J_u$ there exist such matrices) and let 
$\Delta$ be its determinant. Then we have the following equalities of 
submodules of $M$:
\begin{align*}
 (J_xM:J_u)=(J_x+(\Delta)) M \\
(J_xM: (J_x+(\Delta))) =J_u M
\end{align*}
and we have equalities of ideals of $A$:
\begin{align*}
 \Ann_A(J_uM/J_xM)= \Ann_A(M/(J_x+(\Delta))M) \\
\Ann_A((J_x+(\Delta))M/ J_xM) =\Ann_A(M/J_u M)
\end{align*}

\item If moreover $\Ann_A(M/J_xM)=J_x$ (i.e. $M/J_xM$ is a faithful $A/J_x$-module), then the sequences 
$x$ and $u$ are $A$-independent. In particular:
\begin{align*}
(J_x:J_u)=J_x+(\Delta) &=\Ann_A(M/(J_x+(\Delta))M) \\
(J_x: J_x+(\Delta)) =J_u&=\Ann_A(M/J_u M)
\end{align*}
Moreover the Fitting ideal of $J_u/J_x$ over the ring $\ov{A}=A/J_x$ is 
generated by the image of $\Delta$:
\[
 \Fit_{\ov{A}}(J_u/J_x)=(\ov{\Delta})
\]

\item Assume that $M$ is nonzero and that one of the following 
holds:
\begin{enumerate}
 \item $J_u$ is nilpotent, or
\item $M$ is of finite type over $A$ and $J_u\subset \Jac(A)$, or
\item $\Ann_A(M/J_xM) = J_x$ and $J_u\neq A$.
\end{enumerate}
Then $\Delta \notin J_x$, and $\mu_A(J_u)=\mu_A(J_x)=n$.
\end{enumerate}
\end{thm}
\begin{proof}
Let $W$ be a 
matrix as in (2).
Let us prove (1). Let $\mu=(m_1, \dots, m_n)^T\in M^n$ such that $u.\mu=0$. We 
have to prove that $\mu\in J_uM^n$. For each $\ell$ we denote by $\delta^{\mu}_\ell : 
\Lambda^\ell(A^n) \to (\Lambda^{\ell+1}(A^n))\otimes_A M$ the morphism defined by 
$a\mapsto a\wedge \mu$. More precisely it is defined by $a\mapsto \sum_{i=1}^n 
(a\wedge e_i) \otimes m_i$ where $(e_1, \dots, e_n)$ is the canonical basis of 
$A^n$. This defines a morphism of complexes $K(u)\to K(u,M)$ of degree 1. It is indeed a morphism of complexes because $u.\mu=0$, as the computation below shows. For any $\ell$ and any integers $1\leq i_1<\dots 
<i_\ell \leq n$, with the notations $\eps_j=e_{i_1}\wedge \dots \wedge \widehat{e_{i_j}} \wedge \dots \wedge e_{i_\ell}$ and $\eps=e_{i_1}\wedge \dots
 \wedge e_{i_{\ell}}$, we have 
\begin{align*}
 \delta_{\ell-1}^{\mu}(d_{\ell}^{u}(\eps))
 &= \delta_{\ell-1}^{\mu}
   \left(
     \sum_{j=1}^{\ell}(-1)^{j-1}u_{i_j} \eps_j
   \right) \\
 &=\sum_{j=1}^{\ell} (-1)^{j-1}u_{i_j}
    \left((\eps_j \wedge e_{i_j})\otimes m_{i_j}+\sum_{\substack{\alpha=1 \\ \alpha\notin\{i_1,\dots, i_{\ell}\}}}^n
     (\eps_j
    \wedge e_{\alpha})\otimes m_{\alpha}
    \right)
\end{align*}
On the other hand, we have
\begin{align*}
 d_{\ell+1}^{u,M}(\delta_{\ell}^{\mu}(\eps))
  &= \sum_{\substack{\alpha=1 \\ \alpha\notin\{i_1,\dots, i_{\ell}\}}}^n
    d_{\ell+1}^u(\eps\wedge e_{\alpha})\otimes m_{\alpha} \\
  &= \sum_{\alpha\notin\{i_1,\dots, i_{\ell}\}}
     \left(
    \sum_{j=1}^{\ell} (-1)^{j-1}u_{i_j} (\eps_j\wedge e_{\alpha})\otimes m_{\alpha} +(-1)^{\ell}u_{\alpha}\eps\otimes m_{\alpha}
    \right)
\end{align*}
Combining both equations, we get
\begin{align*}
  \delta_{\ell-1}^{\mu}(d_{\ell}^{u}(\eps))
  - d_{\ell+1}^{u,M}(\delta_{\ell}^{\mu}(\eps))
  &= \sum_{j=1}^{\ell} (-1)^{j-1}u_{i_j}
    (\eps_j \wedge e_{i_j})\otimes m_{i_j}
    -\!\!\!\!\!\!\sum_{\alpha\notin\{i_1,\dots, i_{\ell}\}}\!\!\!\!\!\!
     (-1)^{\ell}u_{\alpha}\eps\otimes m_{\alpha}\\
  &= \sum_{j=1}^{\ell} (-1)^{j-1+\ell-j}u_{i_j}
       \eps\otimes m_{i_j}
     - \!\!\!\!\!\!\sum_{\alpha\notin\{i_1,\dots, i_{\ell}\}}\!\!\!\!\!\!
     (-1)^{\ell}u_{\alpha}\eps\otimes m_{\alpha}\\
  &= (-1)^{\ell-1}\eps\otimes\left(\sum_{\alpha=1}^n u_{\alpha}m_{\alpha}\right) \\
  &= (-1)^{\ell-1}\eps\otimes (u.\mu) \\
  &= 0
\end{align*}
This concludes the proof of the fact that $\delta^{\mu} : K(u)\to K(u,M)$ is a morphism of complexes of degree 1.
The matrix $W$ also defines a morphism of complexes $K(x)\to 
K(u)$ of degree 0. Thus we have a commutative diagram:
\[
 \xymatrix@C=1.43pc{
K_n(x) \ar[r]^{d_n^x}     \ar[d]_{\Delta=\Lambda^nW}      &
K_{n-1}(x) \ar[r]^{d_{n-1}^x} \ar[d]_{\Lambda^{n-1}W}  &
K_{n-2}(x) \ar[r] \ar[d]_{\Lambda^{n-2}W}       &
\dots \ar[r]                &
K_{1}(x) \ar[r]^{d_1^x}   \ar[d]_{W}    &                         
K_{0}(x)   \ar@{=}[d]   \\                          
K_n(u) \ar[r]^{d_n^u}     \ar[d]_{\delta^{\mu}_{n}}      &
K_{n-1}(u) \ar[r]^{d_{n-1}^u} \ar[d]_{\delta^{\mu}_{n-1}}  &
K_{n-2}(u) \ar[r]      \ar[d]_{\delta^{\mu}_{n-2}}       &
\dots  \ar[r]           &
K_{1}(u) \ar[r]^{d_{1}^u}   \ar[d]_{\delta^{\mu}_{1}}    &
K_{0}(u)   \ar[d]_{\delta^{\mu}_{0}}    \\
0 \ar[r]                  &
K_n(u,M) \ar[r]^{d_n^{u,M}}           &
K_{n-1}(u,M) \ar[r]       &
\dots \ar[r]              &
K_{2}(u,M) \ar[r]^{d_2^{u,M}}         &
K_{1}(u,M)       
}
\]
By~\ref{cor:homotopies}, the morphism of complexes from the top row to the bottom row is homotopic to a morphism with values in $J_xK_{\bullet}(u,M)$. In particular, there exists a morphism $h_0 : K_0(x)\to K_2(u,M)$ such that $\delta_0^{\mu}=d_2^{u,M}h_0$ 
modulo $J_x$. It follows that $\mu=\delta_0^{\mu}(1)$ belongs to 
$J_uM^n+J_xM^n=J_uM^n$, as required.

Now let us prove (2). Since $\Delta.\Id=W.(\Com W)^T$ we have $u\Delta=x.\Com 
W^T$ hence $\Delta J_u\subset J_x$. From this we deduce the two inclusions
\begin{align*}
 (J_x+(\Delta)) M\subset (J_xM:J_u)\\
J_u M \subset (J_xM: (J_x+(\Delta))) 
\end{align*}
Let us prove the converse of the second one. Let $m\in M$ such that $\Delta 
m\in J_x M$ and let us prove that $m\in J_uM$.
The multiplication by $m$ gives 
a morphism of complexes $K(u)\to K(u,M)$. We also have as above the morphism $\Lambda W : K(x)\to K(u)$. 
\[
 \xymatrix@C=1.5pc{
K_n(x) \ar[r]^{d_n^x}     \ar[d]_{\Lambda^nW=\Delta}      &
K_{n-1}(x) \ar[r] \ar[d]_{\Lambda^{n-1}W}  &
\dots \ar[r]                &
K_{1}(x) \ar[r]^{d_1^x}   \ar[d]_{W}    &                         
K_{0}(x)   \ar@{=}[d]   \\                          
K_n(u) \ar[r]^{d_n^u}     \ar[d]_m      &
K_{n-1}(u)  \ar[r] \ar[d]_{m}  &
\dots  \ar[r]           &
K_{1}(u) \ar[r]^{d_{1}^u}   \ar[d]_{m}    &
K_{0}(u)   \ar[d]_{m}    \\
K_n(u,M) \ar[r]^{d_n^{u,M}}           &
K_{n-1}(u,M) \ar[r]       &
\dots \ar[r]              &
K_{1}(u,M)  \ar[r]^{d_1^{u,M}}         &   
K_{0}(u,M)
}
\]
By~\ref{cor:homotopies} their composition $K(x)\to K(u,M)$ is homotopic to a morphism with values in $J_xK(u,M)$. In particular, there is a 
morphism $h_0 : K_0(x)\to K_1(u,M)$ such that $m=d_1^{u,M}(h_0(1))$ modulo 
$J_xM$, which proves that $m\in J_uM$. Hence we have the equality
$(J_xM: (J_x+(\Delta))) =J_u M$.

Let us prove the inclusion $ (J_xM:J_u)\subset  (J_x+(\Delta)) M$. Let $m\in M$ 
such that $J_u.m\subset J_xM$ and let us prove that $m\in (J_x+(\Delta)) M$. We first prove by induction on $\ell\in\{0,\dots, n\}$ that there is a morphism $f_{\ell} : 
K_{\ell}(u)\to K_{\ell}(x,M)$ that satisfies the following equalities, modulo $J_xK_{\ell}(x,M)$:
\begin{align}
 f_{\ell}d_{\ell+1}^u&=0  \tag{A} \label{eq:liaison1}  \\
 f_{\ell}\Lambda^{\ell}W&=m.\Id \tag{B} \label{eq:liaison2}
\end{align}
For $\ell=0$ we can take $f_0=m\Id$. Then \eqref{eq:liaison1} is satisfied by our assumption on $m$ and \eqref{eq:liaison2} is tautological. Assume that we have constructed $f_{\ell-1}$ for some integer $1\leq \ell\leq n$. By \eqref{eq:liaison1}, there exist morphisms $\mu_1, \dots, \mu_n : K_{\ell}(u) \to K_{\ell-1}(x,M)$ such that $f_{\ell-1}d_{\ell}^u=\sum_{i=1}^nx_i\mu_i$.
\[
  \xymatrix@C=1.5pc{
K_{\ell+1}(x) \ar[r]^{d_{\ell+1}^x}     \ar[d]_{\Lambda^{\ell+1}W}      &
K_{\ell}(x) \ar[r]^{d_{\ell}^x}   \ar[d]_{\Lambda^{\ell}W}    &                         
K_{\ell-1}(x)   \ar[d]^{\Lambda^{\ell-1}W}    \\                          
K_{\ell+1}(u) \ar[r]^{d_{\ell+1}^u}           &
K_{\ell}(u) \ar[r]^{d_{\ell}^u}   \ar@{.>}[d]_{f_{\ell}}
  \ar[rd]_{\mu_i} &
K_{\ell-1}(u)   \ar[d]^{f_{\ell-1}}    \\
    & K_{\ell}(x,M) & K_{\ell-1}(x,M)
}
\]
Let $\Pc_{\ell}$ denote the set of subsets of $\{1,\dots, n\}$ with cardinality $\ell$.
We define $f_{\ell} : K_{\ell}(u)\to K_{\ell}(x,M)$ to be the morphism whose $J$-th component $f_{\ell}^J : K_{\ell}(u)\to M$, for $J=\{j_1,\dots, j_{\ell}\}\in \Pc_{\ell}$ with $j_1<\dots <j_\ell$, is $(-1)^{\ell-1}$ times the $(J\smallsetminus j_{\ell})$-th component of $\mu_{j_\ell}$. Since $(d^u)^2=0$, we have $\sum_{i=1}^nx_i(\mu_i d_{\ell+1}^u)=0$. Since $x$ is $M$-independent, this implies that $\mu_i d_{\ell+1}^u=0$ modulo $J_xK_{\ell-1}(x,M)$, hence $f_{\ell}$ satisfies~\eqref{eq:liaison1}.
Let us prove that it satisfies \eqref{eq:liaison2}. Let $I, J\in \Pc_{\ell}$. We have to prove that $f_{\ell}^J(\Lambda^{\ell}W(e_I))=\delta_{I,J}m$ modulo $J_xM$, where $\delta_{I,J}=1$ if $I=J$ and $\delta_{I,J}=0$ otherwise. Recall that $(-1)^{\ell-1}f_{\ell}^J$ is the $(J\smallsetminus j_\ell)$-th component of $\mu_{j_\ell}$, where $J=\{j_1,\dots, j_\ell\}$ with $j_1<\dots<j_\ell$. We have
\begin{align*}
 \sum_{i=1}^nx_i\mu_i(\Lambda^{\ell}W(e_I))
   &= f_{\ell-1}d_\ell^u (\Lambda^{\ell}W(e_I)) \\
   &= f_{\ell-1}\Lambda^{\ell-1}W(d_\ell^x(e_I)) \\
   &\equiv m.d_{\ell}^x(e_I) \quad \textrm{ modulo } J_x^2K_{\ell-1}(x,M)\\
   &\equiv \sum_{k=1}^{\ell}(-1)^{k-1}x_{i_k}e_{I\smallsetminus i_k}m \quad \textrm{ modulo } J_x^2K_{\ell-1}(x,M)
\end{align*}
where the first congruence follows from the assumption \eqref{eq:liaison2} for $f_{\ell-1}$ and the second one is by definition of $d_{\ell}^x$. By the $M$-independence of the sequence $x$, it follows that $f_{\ell}^J(\Lambda^\ell W(e_I))=0$ modulo $J_xM$, unless there exists an integer $k$ such that $i_k=j_\ell$ and $I\smallsetminus i_k=J\smallsetminus j_\ell$. These conditions can happen only if $k=\ell$ and $I=J$, and in this case we find that $(-1)^{\ell-1}f_{\ell}^J(\Lambda^\ell W(e_I))=(-1)^{\ell-1}m$ modulo $J_xM$, as desired. This concludes the induction. In particular, for $\ell=n$, this yields $f_n : A\to M$ such that $f_n(\Delta)=f_n(\Lambda^nW(1))\equiv m$ modulo $J_xM$, hence $m\in J_xM+\Delta M$.
\details{Pour voir cette preuve à coups de déterminants, voir les calculs en C6 p.53v à 54v pour le cas $n=3$, puis C11 p.59 à 69v pour le cas général (précédé des cas $n=2$ et $n=3$ réécrits).
}

Since $J_u.(J_x+(\Delta))\subset J_x$ we have the inclusions
\begin{align*}
 \Ann_A(J_uM/J_xM)\supset \Ann_A(M/(J_x+(\Delta))M) \\
\Ann_A((J_x+(\Delta))M/ J_xM)\supset \Ann_A(M/J_u M)
\end{align*}
Let $\lambda \in  \Ann_A(J_uM/J_xM)$. Then $J_u\lambda M 
\subset J_x M$, hence $\lambda M\subset (J_xM:J_u)=(J_x+(\Delta))M$, which 
proves that $\lambda\in \Ann_A(M/(J_x+(\Delta))M)$. Similarly the last equality 
of (2) follows from $(J_xM: (J_x+(\Delta))) =J_u M$.

Let us prove (3). Let $\sum x_ib_i=0$ be a relation in $A$. By the 
$M$-independence of $x$, for any $m\in M$ the relation $\sum x_i(b_im)=0$ 
implies that $b_im\in J_xM$. Hence $b_i\in \Ann_A(M/J_xM)=J_x$, which proves 
that $x$ is $A$-independent. Then $u$ is $A$-independent as well by (1).
The equalities $(J_x:J_u)=J_x+(\Delta)$ and $(J_x: J_x+(\Delta)) =J_u$ follow from (2) applied with $M=A$. The equalities $\Ann_A(M/(J_x+(\Delta))M)=(J_x:J_u)$ and $\Ann_A(M/J_uM)=(J_x: J_x+(\Delta))$ follow from the last equalities of~(2) and the faithfulness of $M/J_xM$. 

It remains to prove that $ \Fit_{\ov{A}}(J_u/J_x)=(\ov{\Delta})$. By definition, 
$\Fit_{\ov{A}}(J_u/J_x)$ is the ideal generated by the determinants~$\ov{\det W'}$, 
where $W'\in M_n(A)$ is a matrix such that the composition $\ov{uW'} : \ov{A}^n 
\to \ov{A}^n \to J_u/J_x$ is zero. In particular $\ov{\Delta}=\ov{\det W}\in 
\Fit_{\ov{A}}(J_u/J_x)$. Conversely, let $W'$ be such a matrix. Since $uW'$ has 
entries in $J_x$, so has $uW'(\Com W')^T=u\det W'$, hence $\det W'\in (J_x: 
J_u)=J_x+(\Delta)$. It follows that $\ov{\det W'}\in \ov{\Delta}$, as required.

Finally let us prove (4). If $\Delta \in J_x$, then $(\Delta)+J_x=J_x$ so that 
using (2) we get
\[
 M=(J_xM:J_x)
=(J_xM:J_x+(\Delta))
= J_uM\, .
\]
If $J_u$ is nilpotent this implies $M=0$. If $J_u$ is included in the 
radical of $A$ and $M$ is of finite type, then by Nakayama's Lemma we also have 
$M=0$. In case (c), by (3) we have $J_u=\Ann_A(M/J_uM)=A$. This yields a contradiction in all three cases, hence $\Delta\notin J_x$. In particular $\Delta\neq 0$. This implies $\mu_A(J_u)=n$, for otherwise we could choose $u$ and $W$ such that $u_n=0$ and $W$ has a zero line.
\end{proof}

\details{Voici une preuve de (1) dans le cas (trivial) où la matrice $W$ est diagonale.
Let $x_1,\dots, x_n\in \mgo_A$ and $\lambda_1,\dots, \lambda_n\in A$. If the 
sequence $(\lambda_1x_1,\dots, \lambda_nx_n)$ is $M$-independent, then the 
sequence $(x_1,\dots, x_n)$ is $M$-independent. Indeed, we may 
assume that $\lambda_1=\dots=\lambda_{n-1}=1$. 
 Let $\sum x_im_i=0$ be a relation in $M$. Then $\sum x_i\lambda_nm_i=0$, hence 
there exist elements $m_i'\in M$ such that $m_n=\sum_{i=1}^{n-1}x_im_i 
+\lambda_nx_nm_n$. Then the relation $\sum x_im_i=0$ becomes
\[
 \sum_{i=1}^{n-1}x_i(m_i+x_nm_i') +\lambda_nx_n(x_nm_n')=0
\]
which proves that $m_i+x_nm_i'\in (x_1,\dots, x_{n-1}, \lambda_nx_n)M$ for any 
$i\leq n-1$, whence $m_i\in (x_1, \dots, x_n)M$, as desired.
}

\begin{remarque}\rm
 Note that by~\ref{thm:freeness_criterion_from_Mindep} below, the assumption $\Ann_A(M/J_xM)=J_x$ in (3) and (4) holds in 
particular if $A$ is local with maximal ideal $\mgo_A=J_u$ and $M$ is nonzero and of finite 
type over~$A$. Note also that the equalities in (3) do not  hold in general, even if $M$ is faithful.
\end{remarque}

\begin{remarque}\rm
 Under the assumptions of~\ref{thm:liaison_suites_Mindep}, if $x$ is regular, then so is the sequence~$u$ (\cite[1.2]{Simon_Strooker_Wiebe}). However, if $x$ is strongly $M$-independent, then $u$ does not need to have the same property, even for $M=A$. For example, let $A=k[x]/(x^8)$. Then $x^4$ is strongly independent but not $x^3$.
\end{remarque}

\details{
\begin{question}\rm
 Let $A$ be a Noetherian local ring and let $u$ be a minimal system of 
generators of $\mgo_A$ (hence an $A$-independent sequence). If $x=uW$ is 
$A$-independent, we have seen that $\Delta\notin J_x$. We may ask for a 
converse 
statement, i.e. can we find natural conditions on $\Delta$ that imply that $x$ 
is $A$-independent. It is not true that $\Delta\notin J_x\Rightarrow x$ is 
$A$-independent. It seems difficult in general to formulate a 
conjecture... However, if we assume that $A$ is Gorenstein, then it might be 
possible that:
\begin{enumerate}
 \item There is in $A$ a minimal $A$-independent ideal $J_m$. Let 
$\Delta_m=\det W$ where $W$ is a matrix such that $uW$ generates $J_m$. (In 
particular $\Delta_m\notin J_m$.)
\item For an arbitrary matrix $W$ with determinant $\Delta$, TFAE:
\begin{enumerate}
 \item $uW$ is $A$-independent
\item $\Delta_m\in (\Delta)$
\item $(J_m:J_u)\subset (\Delta)+J_m$
\item $(J_m:J_u)\subset (J_x:J_u)$
\end{enumerate}
On aurait (a),(b), (c) équivalentes en général, impliquent (d), mais (d) 
implique (a) seulement si $W$ est à coeffients dans $\mgo_A$. 
\end{enumerate}
Voir des calculs à ce sujet dans C6, p.57 à 59.
\end{question}
}

\begin{thm}
 \label{thm:freeness_criterion_from_Mindep}
Let $A$ be a Noetherian local ring and let $M$ be a nonzero $A$-module of 
finite type. Let $I\subset \mgo_A$ be an $M$-independent ideal with $\mu_A(I)\geq \edim(A)$. Then:
\begin{enumerate}
\item $\mu_A(I)=\edim(A)$
\item The ring $A/I$ is a complete intersection of dimension zero.
\item The module $M/I^2M$ is free over $A/I^2$.
\item If $I$ is strongly $M$-independent, then $M$ is free over $A$.
\end{enumerate}
\end{thm}
\begin{proof} Let $x_1,\dots, x_n$ be a minimal system of generators of~$I$, let $u_1, \dots, u_n\in A$ be generators of $\mgo_A$ and let $W$ be a 
square matrix of size $n$ with entries in $A$ such that $uW=x$, where $u$ and 
$x$ respectively denote the row vectors $(u_1, \dots, u_{n})$ and 
$(x_1, \dots,x_n)$. Let $\Delta=\det W$. By~\ref{thm:liaison_suites_Mindep}, (4), we have $\Delta\notin I$ and $n=\edim(A)$. Then $W$ is a Wiebe matrix for the ideal $I$, hence $A/I$ is a complete intersection of dimension zero 
by~\cite[2.7]{Simon_Strooker_Wiebe}.
The fact that $\Delta m\in IM \Rightarrow m\in \mgo_A M$ means that $M/IM$ is weakly torsion-free over $A/I$ in the sense of~\cite{Brochard_Mezard_ConjDeSmit}. 
By~\cite[3.7]{Brochard_Mezard_ConjDeSmit} it is actually free over $A/I$, since the latter ring is Gorenstein.
Now assertions (3) and (4) follow from~\ref{prop:indep_implies_freeness}.
\end{proof}

\subsection{Open questions.} Let $A$ be a Noetherian local ring, $(x_1,\dots, x_n)$ a minimal system of generators of its maximal ideal, and $M$ an $A$-module of finite type. In view of~\ref{prop:Mindep_en_termes_Koszul} and its corollaries, it is tempting to wonder whether the Koszul complex $K(x,M)$ could compute the maximal length of an $M$-independent sequence, as it does for $M$-regular sequences. I do not know whether or not this is the case. 

Similarly, if an $A$-module $M$ has finite projective dimension, I do not know if there always exists an $M$-independent sequence of length at least $\edim(A)-\pd_A(M)$. Such an existence result would be useful to get freeness statements in situations where the ring $A$ is not regular, e.g. this could yield a proof of the freeness part of~\cite[1.2]{Brochard_Iyengar_Khare_A_Freeness_Criterion} analogous to that of~\ref{cor:desmit}.

\begin{exemple}\rm
Let $A=k[|x,y|]/(x^2-y^3)$ and let $M=A/(x)$. Since $x$ is a nonzero divisor on $A$, we have $\pd_A(M)=1$. Since $\edim(A/(x))=1$, by~\ref{thm:freeness_criterion_from_Mindep} all $M$-independent sequences must have length $\leq 1$. Assume that there exists a non trivial strongly $M$-independent sequence $(u)$. Since $A$ is reduced, for any $i\geq 0$ we have $u^i\neq 0$. Let $m\in M$ be such that $um=0$. By $M$-independence, $m\in uM$, hence there exists $m'\in M$ such that $m=um'$. Then $u^2m'=0$. Using the strong $M$-independence and arguing by induction, we see that $m\in \cap_{i\geq 0}u^iM=0$. Hence $u$ is a nonzero divisor on $M$, which is a contradiction since $\prof_A(M)=0$. This gives an example of a module of finite projective dimension, such that the maximal length of a strongly $M$-independent sequence is lower than $\edim(A)-\pd_A(M)$.
\end{exemple}

\section{A freeness criterion without patching for balanced modules}
\label{section:balanced_desmit}

One of the ingredients in the work of Wiles and Taylor-Wiles \cite{TaylorWiles1995Ring,Wiles} on modularity lifting theorems is a  patching  method, that is used to reduce the proof of a commutative algebra result to a  statement about modules over regular local rings, statement that follows from the Auslander-Buchsbaum formula. See e.g. \cite[2.1]{Diamond_The_Taylor_Wiles}. Bart de Smit conjectured that if $A\to B$ is a local homomorphism of Artin local rings of the same embedding dimension, then any $A$-flat $B$-module is also flat as a $B$-module. The conjecture was proved recently, and as explained in~\cite[\S 3]{Brochard_DeSmit2}, this allows one to dispense with patching in Wiles' proof of Fermat's Last Theorem.

Since~\cite{TaylorWiles1995Ring,Wiles}, the patching method has evolved and has been generalized in several directions. For example, Calegari and Geraghty propose in~\cite[2.3]{Calegari_Geraghty} a variant of Wiles' patching method for what they call \emph{balanced modules}. Here is the basic freeness result on which their patching method relies. 

\begin{prop}
 \label{prop:CG_freeness_regular_case} Let $A\to B$ be a local morphism of regular local rings such that $\dim(A)=\dim(B)+1$. Let $M$ be a $B$-module, of finite type over $A$, such that $$\dim_{k}\Tor_A^1(M,k)\leq \dim_k\Tor_A^0(M,k).$$ Then $M$ is free over $B$.
\end{prop}
\begin{proof}
 Here $k$ denotes the residue field of $A$. By assumption, the module $M$ admits a presentation of the form $A^d\to A^d\to M\to 0$, with $d=\mu_A(M)$. Let $K$ be the kernel of the first map and let $\eta$ denote the generic point of $\Spec A$. Since $M$ is of finite type over $A$, we have $\dim_A(M)=\dim_B(M)\leq \dim(B)$, hence $\dim_A(M)<\dim(A)$ and $\eta$ is not in the support of $M$. This means that $M_{\eta}=0$. Applying the functor $(-)\otimes_A\Frac(A)$ to the presentation of $M$ and looking at dimensions over $\Frac(A)$ then yields $K_{\eta}=0$. Since $K\subset A^d$ this implies $K=0$. Hence $\pd_A(M)\leq 1$. The Auslander-Buchsbaum formula then yields $\prof_A(M)\geq \dim(A)-1=\dim(B)$. Hence $\prof_B(M)\geq \dim(B)$ and using the Auslander-Buchsbaum formula again, this implies that $M$ is free over $B$.
\end{proof}

Here dim $\Tor_A^0(M,k)$ is the rank of $M$, i.e. the minimal number of generators of $M$ as an $A$-module, while $\dim_k\Tor_A^1(M,k)$ is heuristically the number of relations, so the result says that $M$ is $B$-free if there are not too many relations compared to $\mu_A(M)$. It seems natural to wonder whether it could be possible to dispense with patching here also. For this we would need a statement analogous to~\ref{prop:CG_freeness_regular_case} above for arbitrary Noetherian local rings and not only for regular ones. This raises the following question: if $A\to B$ is a local homomorphism of Noetherian local rings with $\edim(A)-\edim(B)\geq 0$, and if $M$ is a $B$-module of finite type over~$A$, is it possible to give a freeness criterion for $M$ over $B$ in terms of the numbers $\dim\Tor_A^1(M,k)$ and $\mu_A(M)$? Our main result in this direction is~\ref{thm_VP}. Theorem~\ref{thm:liaison_suites_Mindep} is a key ingredient in its proof. First of all we introduce the following number that will be convenient in what follows. 

\begin{defi}
Let $A$ be a Noetherian local ring with residue field $k$ and let $M$ be a finite type $A$-module. The \emph{torsion ratio} $t_A(M)$ of $M$ is 0 if $M=0$, otherwise it is
\[
 t_A(M)=\frac{\dim_k\Tor_A^1(M,k)}{\dim_k\Tor_A^0(M,k)}.
\]
\end{defi}

\begin{remarque}\rm
 An $A$-module $M$ is balanced in the sense of~\cite[2.1 and 2.2]{Calegari_Geraghty} if and only if $t_A(M)\leq 1$. It is free over~$A$ if and only if $t_A(M)=0$.
\end{remarque}

The next three propositions give basic properties of the torsion ratio that will be useful in the sequel.

\begin{prop}
 \label{prop:calcul_defaut}
Let $A$ be a Noetherian local ring and let $M$ be an $A$-module of finite type. 
Let $\pi : A^r \to M$ be a surjection of $A$-modules with $r=\mu_A(M)$ and let $K=\Ker \pi$ be its kernel. Then there is a canonical isomorphism of $\kappa(A)$-vector spaces
\(
 \Tor_1^A(M,\kappa(A)) \simeq K/\mgo_AK\, ,
\)
and the torsion ratio of $M$ is
 \[
 t_A(M)= \frac{\mu_A(K)}{r}\, .
\]
\end{prop}
\begin{proof}
 The exact sequence $0\to K\to A^r\to M\to 0$ induces an exact sequence:
\[
 0\to \Tor_1^A(M,\kappa(A)) \to K/\mgo_AK \to
\kappa(A)^r \to M/\mgo_AM \to 0.
\]
Since $r=\mu_A(M)$, the morphism $\ov{\pi} : \kappa(A)^r \to M/\mgo_AM$ is an isomorphism and both assertions follow.
\end{proof}

\begin{prop}
 Let $A$ be a Noetherian local ring and $M$ a finite type $A$-module. 
Let $I\subset \mgo_A$ be a proper ideal of $A$ and let $\ov{A}=A/I$ and 
$\ov{M}=M/IM$. Then:
\[
 t_{\ov{A}}(\ov{M})\leq t_A(M).
\]
\end{prop}
\begin{proof}
 Let $\pi : A^r \to M$ be a surjection with $r=\mu_A(M)$ and $K=\Ker \pi$. 
Since $I\subset \mgo_A$ the morphism $\ov{\pi}:=\pi\otimes_A \ov{A}$ still 
induces an isomorphism after reduction modulo $\mgo_A$, hence 
$\mu_{\ov{A}}(\ov{M})=\mu_A(M)$. If $K'=\Ker(\ov{\pi})$ we have a surjection of 
$\ov{A}$-modules $\ov{K} \to K'$, so that
\(
 \mu_{\ov{A}}(K')\leq \mu_{\ov{A}}(\ov{K}) = \mu_A(K).
\)
The result now follows from~\ref{prop:calcul_defaut}.
\end{proof}

\begin{prop}
\label{prop:torsion_ratio_of_quotient}
 Let $A$ be a Noetherian local ring and $M$ an $A$-module of finite type. 
Assume that $\mgo_A^2=0$. Then for any submodule $N\subset \mgo_AM$ we have
\[
 t_A(M/N) \geq t_A(M)\, .
\]
\end{prop}
\begin{proof}
 Let $p : M\to M/N$ be the canonical surjection and let $\pi : A^d\to M$ be a minimal presentation of $M$. Since $N\subset \mgo_AM$, the composition $p\circ \pi$ is a minimal presentation of $M/N$. Let $K=\Ker \pi$ and $\ov{K}=\Ker(p\circ \pi)$. We have $K\subset \ov{K}\subset \mgo_AA^d$. Since $\mgo_A^2=0$, $K$ and $\ov{K}$ are $\kappa(A)$-vector spaces, and the inclusion $K\subset \ov{K}$ proves that $\mu_A(\ov{K})\geq \mu_A(K)$. The result now follows from~\ref{prop:calcul_defaut}.
\end{proof}

\begin{lem}
 \label{lemme:defaut_negatif_ameliore}
Let $A$ be a Noetherian local ring and let $M$ be an $A$-module of finite type. 
Let $\delta$ be a nonnegative integer.  
Let
\[
 \xymatrix{
A^{t} \ar[r]^D & A^r \ar[r]^{\pi} &M \ar[r] &0
}
\]
be a minimal presentation of $M$. Let $(x_1, \dots, x_n)$ be a minimal system 
of generators of $\mgo_A$ and let us decompose the matrix $D$ as
\[
 D=x_1D_1+\dots+x_nD_n
\]
with $D_i\in M_{rt}(A)$ (this is possible because $D$ has entries in 
$\mgo_A$, since the presentation is minimal). The following are equivalent:
\begin{enumerate}
 \item $\mgo_AM=(x_{\delta+1}, x_{\delta+2},\dots, x_n)M$
\item The morphism $(\ov{D_1}, \dots, \ov{D_{\delta}}) : \kappa(A)^t \to 
\kappa(A)^{r\delta}$ is surjective, where $\ov{D_i}\in 
M_{rt}(\kappa(A))$ denotes the reduction of $D_i$ 
modulo $\mgo_A$.
\end{enumerate}
In particular, when the conditions are satisfied, then $\delta r \leq t$, hence if $M\neq 0$ then
\[t_A(M)\geq \delta,\]
with equality if and only if $(\ov{D_1}, 
\dots, \ov{D_{\delta}})$ is an isomorphism.
\end{lem}
\begin{proof}
 Assume (1) and let $f_1, \dots, f_{\delta}\in A^r$. Since 
$x_1\pi(f_1)+\dots+x_{\delta}\pi(f_{\delta})$ belongs to $\mgo_AM=
(x_{\delta+1}, \dots, x_n)M$, there exist elements $f_{\delta+1}, \dots, f_n\in 
A^r$ such that 
$\sum x_i \pi(f_i)=0=\pi(\sum x_if_i)$. Hence there is an 
element $\lambda \in A^{t}$ such that $\sum x_if_i =D\lambda=\sum 
x_iD_i\lambda$. Then 
\[
 \sum_{i=1}^n x_i(f_i-D_i\lambda)=0
\]
so that $f_i-D_i\lambda\in \mgo_A.A^r$ for any $i$, by independence of~$x_1,\dots, x_n$. Hence 
$\ov{f_i}=\ov{D_i} \ov{\lambda}$ in $\kappa(A)^r$ and this proves that 
$(\ov{D_1}, \dots, \ov{D_{\delta}})$ is surjective.

Conversely assume (2).
Let $f_1, \dots, f_{\delta}\in A^r$. By assumption there exists $\lambda\in 
A^t$ such that modulo $\mgo_A$ we have the equalities
\begin{align*}
 D_1\lambda&=f_1 \\
\vdots \\
D_{\delta}\lambda&=f_{\delta}
\end{align*}
Since $\sum x_i 
\pi(D_i\lambda)=\pi(D\lambda)=0$, we get:
\begin{align*}
\sum_{i=1}^{\delta}x_i\pi(f_i)&=\sum_{i=1}^{\delta} 
x_i\pi(f_i-D_i\lambda)+\sum_{i=1}^{\delta}x_i\pi(D_i\lambda) \\
&= \sum_{i=1}^{\delta}x_i\pi(f_i-D_i\lambda)-\sum_{i=\delta+1}^n 
x_i\pi(D_i\lambda)\\
& \in \mgo_A^2M +(x_{\delta+1}, \dots, x_n)M
\end{align*}
This proves that
\[
 \mgo_AM \subset \mgo_A^2M+(x_{\delta+1}, \dots, x_n)M\, .
\]
By Nakayama this implies $\mgo_AM\subset (x_{\delta+1}, \dots, x_n)M$.
\end{proof}

\begin{cor}
\label{cor:defaut_negatif}
 Let $\varphi : A\to B$ be a local morphism of Noetherian local rings and let~$M$ be a nonzero $B$-module of finite type over $A$. Then $$t_A(M)\geq \edim(A)-\mu_B(\mgo_AB).$$ In particular, if $B/\mgo_AB$ is a complete intersection, then
 \[
  t_A(M)\geq \edim(A)-\edim(B).
 \]
\end{cor}

\begin{exemple}\rm
In general we do not have upper bounds on $t_A(M)$. For example, let $A=k[|x,y|]$ and let $M=A/\mgo^n$. Then $\mu_A(M)=1$, hence
\(
 t_A(M)=\mu_A(\mgo_A^n)=n+1.
\)
\end{exemple}

\begin{exemple}\rm
\label{exemple:defaut_positif}
The second inequality of~\ref{cor:defaut_negatif} can fail if $B/\mgo_AB$ is not a complete intersection, even if $B$ is finite. For example let $k$ be a field and let 
$A=k[x,y,z]/(x,y,z)^2$. Let
\[
 B=A[u,v]/(u^2-x, uv-y, v^2-z)=k[u,v]/(u,v)^4.
\]
Then $(1, u, v)$ is a minimal system of generators for $B$ as an $A$-module, 
and the kernel of the corresponding presentation $\pi : A^3\to B$ is 
(minimally) generated by the two relations $yu=xv$ and $zu=yv$. Hence 
$t_A(B)=\frac 23$, but $\edim(A)-\edim(B)=1$. The second inequality of~\ref{cor:defaut_negatif} can also hold for any finite $B$-module without $B/\mgo_AB$ being a complete intersection. For example with $A\to B$ as 
above, let $B'=B/(v^3)$. Then $B'/\mgo_AB'$ is not a complete intersection, but $t_A(M)\geq 1$ for any finite $B'$-module $M$, with equality if and only if $M$ is $B'$-free. Indeed, $(1, u, v)$ is a minimal system of generators for $B'$ as an $A$-module and the relations are generated by the two above and $zv=0$, hence we have $t_A(B')=\frac 33=1$. From this we have $t_A(M)=1$ for any finite free $B'$-module. It remains to show that if $M$ is not $B'$-free, then $t_A(M)>1$. For this we only sketch the argument because the proof relies on elementary but a bit tedious computations. We can assume that $M=B^{n}/F$ for some sub-$B$-module $F$ of $B^{n}$, with $n=\mu_B(M)$. By~\ref{prop:torsion_ratio_of_quotient} and~\ref{prop:free_on_specialfiber_implies_free} we may assume that $F\nsubseteq \mgo_A {B^n}$. Let $I\subset B$ denote the projection of $F$ onto the first factor of $B^n$. We have $I\subset \mgo_B$ and we may assume that $I\nsubseteq \mgo_AB$. Then, using an explicit description of what the ideal $I$ can be, and arguing by induction on $n$, it is possible to give a minimal presentation of $M$ for which we can construct enough independent relations, and finally prove that $t_A(M)>1$.
\details{Voir C5 p.61 pour ces deux exemples.}
\end{exemple}

\begin{prop}
 \label{prop:critere_suite_Mindep_pour_modules_eq}
Let $A$ be a Noetherian local ring and let $M$ be an $A$-module of finite type. 
Assume that $t_A(M)\leq \delta$ and that $\mgo_AM=(x_{\delta+1}, \dots, 
x_n)M$ for some minimal system of generators $(x_1, \dots, x_n)$ of $\mgo_A$. 
Then the sequence $(x_{\delta+1}, \dots, x_n)$ is $M$-independent.
\end{prop}
\begin{proof}
We keep the notations of~\ref{lemme:defaut_negatif_ameliore} for the minimal presentation of $M$. Let $\sum_{i=\delta+1}^nx_im_i=0$ be a relation in $M$. Let $f_i\in A^r$ be 
such that $\pi(f_i)=m_i$. Then we have $\sum x_if_i\in \Ker \pi$, hence 
there exists $\lambda \in A^t$ such that
\[
 \sum_{i=\delta+1}^nx_if_i=D\lambda=\sum_{i=1}^n x_iD_i\lambda\, .
\]
By independence of $x_1,\dots, x_n$ on $A^r$, the elements $D_i\lambda$, for $i\leq \delta$, and 
$f_i-D_i\lambda$ for $i\geq \delta+1$, are in $\mgo_A.A^r$. Since 
by~\ref{lemme:defaut_negatif_ameliore} $(\ov{D_1},\dots, \ov{D_{\delta}})$ is 
an isomorphism, 
this implies that $\lambda\in \mgo_A.A^t$, and in turn $f_i\in \mgo_A.A^r$ for 
$i\geq \delta+1$. Then 
$m_i\in \mgo_AM=(x_{\delta+1}, \dots, x_n)M$ which proves that $(x_{\delta+1}, 
\dots, x_n)$ is $M$-independent.
\end{proof}

In the following Theorem, note that if we assume that $B/\mgo_AB$ is a complete intersection, then by~\cite[5.4]{Brochard_Mezard_ConjDeSmit} we have~$\mu_B(\mgo_AB)\leq \edim(B)$, hence any finite $B$-module~$M$ such that $t_A(M)\leq \edim(A)-\edim(B)$ satisfies the assumptions.

\begin{thm}
 \label{thm_VP}
Let $\varphi : A\to B$ be a finite local morphism of Noetherian local rings and~$M$ a finite $B$-module. Assume that there exists an integer $\delta\geq 0$ such that $t_A(M)\leq \delta\leq \edim(A)-\edim(B)$ and 
$\mgo_AM=(x_{\delta+1}, \dots, x_n)M$ for some minimal system of generators $(x_1, \dots, x_n)$ of $\mgo_A$. Then: 
\begin{enumerate}
 \item $M$ is free over $B$.
\item If $M\neq 0$, then 
$B/\mgo_A B$ is a complete intersection of dimension 0 and 
we have the equalities:
$$t_A(M)= \delta= \edim(A)-\edim(B)=t_A(B)$$
\details{
On a aussi :
\begin{align*}
& \edim(B)=\mu_B(\mgo_AB) \\
& t_{A}(M)=t_A(B)=t_{\ov{A}}(\ov{B})=\delta
\end{align*}
where $\ov{A}=A/\mgo_A^2$ and $\ov{B}=B/\mgo_A^2B$.
}
\end{enumerate}
\end{thm}
\begin{proof}
We may assume that $M\neq 0$. 
By~\ref{prop:critere_suite_Mindep_pour_modules_eq} the sequence $(x_{\delta+1}, 
\dots, x_n)$ is $M$-independent. But it has length $n-\delta$ and by assumption 
$\edim(B)\leq n-\delta$. Then by~\ref{thm:freeness_criterion_from_Mindep} we 
already know that $B/\mgo_A B$ is a complete intersection of dimension 0, that 
$\edim(B)=\edim(A)-\delta$ and that $M/\mgo_AM$ is free over 
$B/\mgo_AB$. By~\cite[5.4]{Brochard_Mezard_ConjDeSmit} we have 
$\mu_B(\mgo_AB)\leq \edim(B)$. By~\ref{thm:liaison_suites_Mindep} applied to the ideals $(x_{\delta+1},\dots, x_n)B$ and~$\mgo_AB$ of $B$, we get $\mu_B(\mgo_AB)=n-\delta$. By~\ref{cor:defaut_negatif}, 
$t_A(M)=\delta$ and $t_A(B)\geq \delta$. 

 Let $m_1, \dots, m_{\ell}$ be elements of $M$ the images of which form a 
$\kappa(B)$-basis of $M/\mgo_BM$. Let $p : B^{\ell}\to M$ be the corresponding 
surjective morphism. Since $M_0=M/\mgo_AM$ is free over 
$B_0=B/\mgo_A B$, the morphism $p$ induces an isomorphism $\ov{p} : B_0^{\ell} 
\to M_0$. In particular, $\Ker p \subset \mgo_AB^{\ell}$. Let us also choose a 
presentation of $B^{\ell}$ as an $A$-module, say $\pi : A^r \to B^{\ell}$, with 
$r$ minimal, i.e. $\pi$ 
induces an isomorphism $\ov{\pi} : \kappa(A)^r \to B_0^{\ell}$. In particular 
$p\circ \pi$ induces an isomorphism $\kappa(A)^r \to M_0$, hence $r$ is also 
the minimal number of generators for $M$ as an $A$-module. Let $K=\Ker \pi$ and 
$K'=\Ker(p\circ \pi)$. We have $K\subset K'\subset \mgo_A A^r$.
Since $\mu_B(\mgo_AB)=n-\delta$ there exists a minimal system of generators 
$(x_1,\dots, x_n)$ of $\mgo_A$ such that $x_{\delta+1}, 
\dots, x_n$ generate $\mgo_AB$. 

For each $i\leq \delta$, we choose a relation
\[
 \varphi(x_i)=\sum_{j=\delta+1}^nb_{ij}\varphi(x_j)
\]
in $B$. Let $(e_k)_{k=1\dots r}$ denote the canonical basis of $A^r$ and let us 
choose a lift $\beta_{ij}^k$ of $b_{ij}\pi(e_k)$ in $A^r$. Now for $1\leq i\leq 
\delta$ and $1\leq k\leq r$ let
\[
 g_{ik}:=x_ie_k-\sum_{j=\delta+1}^nx_j\beta_{ij}^k \in A^r.
\]
We have $\pi(g_{ik})=0$ hence $g_{ik}\in \Ker \pi=K$. Let $(f_{ik}), 
i\in\{1\dots\delta\}, k\in\{1\dots, r\}$ denote the canonical basis of 
$A^{\delta r}$ and let $D : A^{\delta r} \to A^r$ defined by $D(f_{ik})=g_{ik}$. 
We can write $D=\sum x_jD_j$ with $D_j$ defined by $D_j(f_{ik})=e_k$ if 
$i=j\leq \delta$, $D_j(f_{ik})=0$ if $j\leq \delta$ and $j\neq i$, and 
$D_j(f_{ik})=-\beta_{ij}^k$ if $j\geq \delta+1$.

On the other hand, we have $t_A(M)=\delta$. Let $D' : A^{\delta r} \to 
A^r$ be a minimal morphism with image $K'$. There exist a decomposition $D' 
=\sum x_iD'_i$. Since $K\subset K'$, there exists a square matrix $P\in 
M_{\delta r}(A)$ such 
that $D=D'P$. Then $\sum x_iD_i =\sum x_iD'_iP$ and it follows 
that $D_i=D'_iP$ modulo $\mgo_A$ for any $1\leq i\leq n$. Then 
$\ov{D_i}=\ov{D'_iP}$ for all $i$ and in particular we have the equality of 
square matrices:
\[
\begin{pmatrix}
 \ov{D_1} \\
\vdots \\
\ov{D_{\delta}}
\end{pmatrix} =
\begin{pmatrix}
 \ov{D'_1}\\
\vdots\\
\ov{D'_{\delta}}
\end{pmatrix}
P.
\]
But the left-hand side is the identity, hence $\ov{P}$ is invertible. Then 
$\det(\ov{P})\neq 0$ in $\kappa(A)$, so that $\det(P)\in A^\times$ and $P$ is 
invertible. This proves that $K=K'$. Hence $B^{\ell} \to M$ is an 
isomorphism and $M$ is free over $B$. Lastly, since $M$ is $B$-free and nonzero, we have $t_A(M)=t_A(B)$.
\end{proof}

\begin{remarque}\rm
 If $\mgo_A^2=0$ the inclusion $K\subset K'$ is an inclusion of 
$\kappa(A)$-vector spaces with $\dim K \geq d$ and $\dim K' \leq d$ so that the 
equality $K=K'$ is obvious and the proof is substantially simplified. Actually 
with the arguments in the second part of the above proof we get the following. (There is also a variant where the assumption $\mgo_A^2=0$ is replaced with the assumption that, with the notations of the proof of~\ref{thm_VP}, $\cap_{i=1}^\delta \Ker \ov{D_i}=0$. Indeed, this assumption also implies that $P$ is invertible and the proof flaws as in~\ref{thm_VP}.)
\end{remarque}

\begin{prop}
 \label{prop:free_on_specialfiber_implies_free}
Let $A\to B$ be a finite local morphism of Noetherian local rings with $\mgo_A^2=0$ and let $M$ be a finite $B$-module. Assume that
\[
 t_A(M)\leq t_A(B)
\]
and that $M/\mgo_AM$ is free over 
$B/\mgo_A B$. Then $M$ is free over $B$.
\end{prop}

\begin{exemple}\rm
 \label{ex:free_sur_spe_nimplique_pas_free}
 This does not hold without the assumption that $\mgo_A^2=0$. For example let $A$ be the quotient of $k[x,y,z]$ by $(x,y,z)^2+(y,z)^2+(xz)$ and let $B$ be the quotient of $k[u,v]$ by $(u,v)^6+(v^2,uv)^2$. Let $A\to B$ be the morphism that maps $x, y, z$ to $u^2, uv, v^2$ respectively and let $M=B/(uv^2)$. Then $t_A(M)=t_A(B)=1$ and $M/\mgo_AM$ is free over $B/\mgo_AB$ but $M$ is not free over $B$.  
\end{exemple}

\bibliographystyle{plain}
\addcontentsline{toc}{section}{References}
\bibliography{../../../texmf/tex/latex/monlatex/mabiblio}

\end{document}